\documentclass[12pt,a4paper]{amsart}

\usepackage[T1]{fontenc}
\usepackage[utf8]{inputenc}
\usepackage[english]{babel}
\usepackage{amsmath,amssymb,amsthm}
\usepackage{mathrsfs}
\usepackage{enumerate}
\usepackage[hidelinks]{hyperref}
\usepackage[margin=3cm]{geometry}

\newtheorem{proposition}{Proposition}[section]
\newtheorem{lemma}[proposition]{Lemma}
\newtheorem{corollary}[proposition]{Corollary}

\theoremstyle{definition}
\newtheorem{definition}[proposition]{Definition}
\newtheorem{example}[proposition]{Example}
\newtheorem{notation}[proposition]{Notation}

\theoremstyle{remark}
\newtheorem{remark}[proposition]{Remark}

\title{The multiples of a numerical semigroup}

\author{Ignacio Ojeda}
\address{Department of Mathematics,
University of Extremadura,
06071 Badajoz, Spain}
\email{ojedamc@unex.es}

\author{Jos\'e Carlos Rosales}
\address{Department of Algebra,
University of Granada, Granada, Spain}

\subjclass[2010]{Primary 20M14;
Secondary 05C05, 11D07}

\keywords{Numerical semigroup, monoids,
quotient of a numerical semigroup,
irreducible numerical semigroups, rooted tree}

\thanks{This work is licensed under a Creative
Commons Attribution 4.0 International License.}

\begin{document}

\maketitle

\begin{abstract}
Given two numerical semigroups $S$ and $T$ we say that $T$ is a multiple of $S$ if there exists an integer $d \in \mathbb{N} \setminus \{0\}$ such that $S = \{x \in \mathbb{N} \mid d x \in T\}$. In this paper we study the family of multiples of a (fixed) numerical semigroup. We also address the open problem of finding numerical semigroups of embedding dimension $e$ without any quotient of embedding dimension less than $e$, and provide new families with this property.

\keywords{Numerical semigroup, monoids, quotient of a numerical semigroup, irreducible numerical semigroups, rooted tree.}
\end{abstract}

\section{Introduction}

Let $\mathbb{Z}$ be the set of integers and $\mathbb{N} = \{x \in \mathbb{Z} \mid x \geq 0\}$. A \emph{submonoid} of $(\mathbb{N},+)$ is a subset of $\mathbb{N}$ containing $0$ and closed under addition. 

If $A$ is a non-empty subset of $\mathbb{N}$, then we write $\langle A \rangle$ for the submonoid of $(\mathbb{N},+)$ generated by $A$, that is, 
\[\langle A \rangle = \left\{u_1 a_1 + \cdots + u_n a_n \mid n \in \mathbb{N} \setminus \{0\},\{a_1, \ldots, a_n\} \subseteq A\ \text{and}\ \{u_1, \ldots, u_n\} \subset \mathbb{N} \right\}.\] If $M$ is a submonoid of $(\mathbb{N}, +)$ and $M = \langle A \rangle$ for some $A \subseteq \mathbb{N}$, we say that $A$ is a \emph{system of generators} of $M$. Moreover, if $M \neq \langle B \rangle$ for every $B \subsetneq A$, then we say that $A$ is a \emph{minimal system of generators} of $M$. In \cite[Corollary 2.8]{libro} it is shown that every submonoid of $(\mathbb{N},+)$ has a unique minimal system of generators which is also finite. We write $\operatorname{msg}(M)$ for the minimal system of generators of $M$. The cardinality of $\operatorname{msg}(M)$ is called the \emph{embedding dimension} of $M$ and is denoted $\operatorname{e}(M)$.

A \emph{numerical semigroup}, $S$, is a submonoid of $(\mathbb{N},+)$ such that $\mathbb{N} \setminus S$ is finite. The set of numerical semigroups is denoted by $\mathscr{S}$. In \cite[Lemma 2.1]{libro} it is shown that a submonoid $M$ of $(\mathbb{N}, +)$ is a numerical semigroup if and only if $\gcd(\operatorname{msg}(M)) = 1$. 

If $S$ is a numerical semigroup, then both $\operatorname{F}(S) = \max(\mathbb{Z} \setminus S)$ and $\operatorname{g}(S) = \#(\mathbb{N} \setminus S)$, where $\#$ stands for cardinality, are two important invariants of $S$ called \emph{Frobenius number} and \emph{genus}, respectively.

The Frobenius problem consists in finding formulas for the Frobenius number and the genus of a numerical semigroup in terms of its minimal system of generators (see \cite{alfonsin}). This problem is solved in \cite{sylvester} for numerical semigroups of embedding dimension equal to two, remaining open, in general, for numerical semigroups of embedding dimension greater than or equal to three.
 
Let $T$ be a numerical semigroup. Given $d \in \mathbb{N} \setminus \{0\}$, we write \[\frac{T}d = \{x \in \mathbb{N} \mid d\, x \in T\}.\]  In \cite{proportional} it is proved that $\frac{T}{d}$ is a numerical semigroup. This semigroup is called the \emph{quotient of $T$ by $d$}.
In this case, we also say that $T$ is a \emph{$d-$multiple of $\frac{T}{d}$}. 

Correspondingly, given two numerical semigroups $S$ and $T$, we say that $T$ is a multiple of $S$ if there exists $d \in \mathbb{N} \setminus \{0\}$ such that $\frac{T}d = S$; in particular, $S$ is an arithmetic extension of $T$ (see \cite{ojeda}). 

This paper is devoted to the study of multiples of a numerical semigroup. The interest of this study is partially motivated by the problem that we present below. In \cite{full} it is shown that a numerical semigroup is proportionally modular if and only if it has a multiple of embedding dimension two; this result together with those in \cite{theset} allow us to give an algorithmic procedure for deciding whether a numerical semigroup has a multiple of embedding dimension two. In \cite{CIM} the problem of finding a numerical semigroup that has no multiples of embedding dimension three arises. In \cite{canadian} the existence of such semigroups is proved although no concrete example is given. Recently, in \cite[Theorem 3.1]{australia}, examples of numerical semigroups that have no multiples of embedding dimension $k$ are given for (fixed) $k \geq 3$.

Today it is still an open problem to give an algorithmic procedure that determines whether a given numerical semigroup has a multiple with embedding dimension greater than or equal to three. This problem is part of the more ambitious goal of determining \[\min\{\operatorname{e}(T) \mid T\ \text{is a multiple of}\ S\}\] for a given numerical semigroup $S$. This number is called \emph{quotient rank of $S$} (see \cite{australia}). It is clearly bounded above by $\operatorname{e}(S)$, since $S/1 = S$. Consequently, $S$ is said to have \emph{full quotient rank} when the quotient rank of $S$ is equal to $\operatorname{e}(S)$.

Let $S$ be a numerical semigroup and $d \in \mathbb{N} \setminus \{0\}$. We define \[\operatorname{M}_d(S) := \left\{T \in \mathscr{S} \mid \frac{T}d = S\right\}\] and write $\max \operatorname{M}_d(S)$ for the set of maximal elements of $\operatorname{M}_d(S)$ with respect to the inclusion. 

In Section \ref{Sect2}, we prove that $\max \operatorname{M}_d(S)$ is finite. Moreover, we prove that there is a surjective map $\Theta_S^d : \operatorname{M}_d(S) \to \max \operatorname{M}_d(S)$. Thus, to compute $\operatorname{M}_d(S)$ it is enough to determine the fibers of $\Theta_S^d$ as long as $\max \operatorname{M}_d(S)$ is known. The task of computing $ \max \operatorname{M}_d(S)$ can be performed following \cite{aequationes}, the details are briefly discussed in Section \ref{Sect2} and a draft code for computing $ \max \operatorname{M}_d(S)$ is also provided. We finish this section by showing that the necessary and sufficient condition for $S$ to be irreducible is that all elements of $ \max \operatorname{M}_d(S)$ are irreducible (independently of $d$, in fact).

In Section \ref{Sect3}, we deal with the fibers of the map $\Theta_S^d$. Although the number of fibers is finite, their cardinality is not necessarily finite (see \cite{swanson}). So, we arrange the elements in each fiber of $\Theta_S^d$ is form of a rooted tree with root the corresponding element in $\max \operatorname{M}_d(S)$. This tree is locally finite, that is, the number of children of a given node is finite. Then, after characterizing the children of a given node, we can formulate a recursive process that will allow us to theoretically compute the elements of a fiber of $\Theta_S^d$ from the corresponding element in $\max \operatorname{M}_d(S)$. We also show that leaves (childless nodes) and fibers of cardinality one can appear. We emphasize that our recursive process can be easily truncated so as not to exceed some given Frobenius number or genus, thus producing a true algorithm.

Finally, in section \ref{Sect6} we give a sufficient condition for a numerical semigroup to have full quotient rank. Moreover, we prove that numerical semigroups with unique Betti element (and more generally, universally free numerical semigroups) have full quotient rank.

\section{The set $\max \operatorname{M}_d(S) $}\label{Sect2}

\begin{definition}
Let $S$ and $T$ be numerical semigroups and $d \in \mathbb{N} \setminus \{0\}$. We say that $T$ is a \emph{$d-$multiple of $S$} if $\frac{T}d = S$.
\end{definition}

Alternatively, we have the following equivalent and easy-to-prove definition that one semigroup is a multiple of another.

\begin{lemma}\label{lema2}
Let $S$ and $T$ be two numerical semigroups. Given $d \in \mathbb{N} \setminus \{0\}$, one has that $T$ is $d-$multiple of $S$ if and only if \[d\, (\mathbb{N} \setminus S) \subseteq \mathbb{N} \setminus T \subseteq \mathbb{N} \setminus d\, S.\] Specifically, in this case, $\operatorname{F}(T) \geq d \operatorname{F}(S)$.
\end{lemma}

Let $S$ and $T$ be two numerical semigroups. Given $d \in \mathbb{N} \setminus \{0\}$, we write $\operatorname{M}_d(S)$ for the set of all numerical semigroups $d-$multiples of $S$, which, by Lemma \ref{lema2}, is equal to \[\left\{T \in \mathscr{S}\  \mid\ d\, (\mathbb{N} \setminus S) \subseteq \mathbb{N} \setminus T \subseteq \mathbb{N} \setminus d\, S \right\}.\] It is known that there are infinitely many elements in the set $\operatorname{M}_d(S)$ (see, e.g., \cite{swanson}).

\begin{notation}
Given $S \in \mathscr{S}$ and $d \in \mathbb{N} \setminus \{0\}$, we write $\max \operatorname{M}_d(S) $ for the set of maximal elements of $\operatorname{M}_d(S)$ with respect to inclusion.  
\end{notation}

Let us prove that $\max \operatorname{M}_d(S) $ is a non-empty set with finitely many elements, for every $S \in \mathscr{S}$ and $d \in \mathbb{N} \setminus \{0\}$. But first, we need to see some previous results on $\operatorname{M}_d(S)$.

\begin{proposition}\label{prop3}
Let $S$ be a numerical semigroup and $d \in \mathbb{N} \setminus \{0\}$. If $\{T, T'\} \subseteq \operatorname{M}_d(S), T \subsetneq T'$ and $f = \max(T' \setminus T)$, then $T \cup \{f\} \in \operatorname{M}_d(S)$.
\end{proposition}

\begin{proof}
By \cite[Lemma 4.35]{libro}, we have that $T \cup \{f\}$ is a numerical semigroup. So, it suffices to see that $\frac{T}d = S$. Now, since $T \subsetneq T \cup \{f\} \subseteq T'$, we have that $\mathbb{N} \setminus T' \subseteq \mathbb{N} \setminus (T \cup \{f\}) \subseteq \mathbb{N} \setminus T$. Thus, by Lemma \ref{lema2}, we conclude that \[S = \frac{T}d \subseteq \frac{T \cup \{f\}}d \subseteq \frac{T'}d = S\] and we are done.
\end{proof}

\begin{lemma}\label{lema7}
Let $S \neq \mathbb{N}$ be a numerical semigroup and $d \in \mathbb{N} \setminus \{0\}$. If $T \in \operatorname{M}_d(S)$ and $T \not\in \max \operatorname{M}_d(S) $, then there exists \begin{equation}\label{theta}\theta_S^d(T) := \max\left\{ x \in \mathbb{N} \setminus T \mid T \cup \{x\} \in \operatorname{M}_d(S)\right\}.\end{equation}
\end{lemma}

\begin{proof}
Since $T \not\in \max \operatorname{M}_d(S) ,$ there exists $T' \in \operatorname{M}_d(S)$ such that $T \subsetneq T'$. Let $f = \max(T' \setminus T)$, by Proposition \ref{prop3}, we have that $f \in \left\{ x \in \mathbb{N} \setminus T \mid T \cup \{x\} \in \operatorname{M}_d(S)\right\}$. So, $\left\{ x \in \mathbb{N} \setminus T \mid T \cup \{x\} \in \operatorname{M}_d(S)\right\}$ is non-empty and it has a maximum because it is a subset of $\mathbb{N} \setminus T$ which is a finite set.
\end{proof}

Observe that, by Lemma \ref{lema2}, we have that $\theta_S^d(T) \not\in d\, S,$ for every $T \in \operatorname{M}_d(S), S \in \mathscr{S}$ and $d \in \mathbb{N} \setminus \{0\}$. In the next section we will discuss the properties of $\theta_S^d(-)$ in more detail.

\begin{proposition}
The set $\max \operatorname{M}_d(S) $ is not empty, for every $S \in \mathscr{S}$ and $d \in \mathbb{N} \setminus \{0\}$.
\end{proposition}

\begin{proof}
Let $S \in \mathscr{S}$. If $T \in \operatorname{M}_d(S)$, then we define the following sequence: $T_0 = T$ and \begin{equation}\label{ecu1} T_{i+1} = \left\{\begin{array}{cc} T_i \cup \{\theta_S^d(T_i)\} & \text{if}\ T_i \not\in \max \operatorname{M}_d(S) ; \smallskip \\T_i & \text{otherwise}, \end{array} \right.\end{equation} for each $i \in \mathbb{N}$. Then, since $T = T_0 \subseteq \ldots \subseteq T_i \subseteq \ldots $ and $\mathbb{N} \setminus T$ is finite, we conclude that there exits $k \in \mathbb{N}$ such that $T_k = T_{k+i},$ for every $i \in \mathbb{N}$. Necessarily, by construction, $T_k \in \max \operatorname{M}_d(S) $.
\end{proof}

Let us prove now that $\max \operatorname{M}_d(S) $ has finite cardinality for every $S \in \mathscr{S}$ and $d \in \mathbb{N} \setminus \{0\}$.

\begin{lemma}\label{lema4}
Let $S$ be a numerical semigroup and $d \in \mathbb{N} \setminus \{0\}$. If $T \in \operatorname{M}_d(S)$, then $T \cup \{d \operatorname{F}(S) + i \mid i \in \mathbb{N} \setminus \{0\} \} \in \operatorname{M}_d(S)$. 
\end{lemma}

\begin{proof}
Clearly, $T' = T \cup \{d \operatorname{F}(S) + i \mid i \in \mathbb{N} \setminus \{0\} \}$ is a numerical semigroup. Now, since $\operatorname{F}(S) = \max(\mathbb{N} \setminus S)$, we have that $d\,(\mathbb{N} \setminus S) \subseteq \mathbb{N} \setminus \{d \operatorname{F}(S) + i \mid i \in \mathbb{N} \setminus \{0\} \}$. Therefore, by Lemma \ref{lema2}, we obtain that $\mathbb{N} \setminus T' \subseteq \mathbb{N} \setminus T \subseteq \mathbb{N} \setminus d\, S$ and that \begin{align*}
\mathbb{N} \setminus T' & = \mathbb{N} \setminus (T \cup \{d \operatorname{F}(S) + i \mid i \in \mathbb{N} \setminus \{0\} \})\\ & = (\mathbb{N} \setminus T) \cap  (\mathbb{N} \setminus \{d \operatorname{F}(S) + i \mid i \in \mathbb{N} \setminus \{0\} \}) \supseteq d\,(\mathbb{N} \setminus S).  
\end{align*}
Now, by Lemma \ref{lema2} again, we conclude that $T' \in \operatorname{M}_d(S)$.
\end{proof}

\begin{proposition}\label{prop8I}
Let $S \in \mathscr{S} \setminus \mathbb{N}$ and $d \in \mathbb{N} \setminus \{0\}$. If $T \in \max \operatorname{M}_d(S) $ then $\operatorname{F}(T) = d \operatorname{F}(S)$. In particular, $\max \operatorname{M}_d(S) $ has finite cardinality.
\end{proposition}

\begin{proof}
If $T \in \max \operatorname{M}_d(S) $, by Lemma \ref{lema4}, we have that $\mathbb{N} \setminus T \subseteq \{1, \ldots, d \operatorname{F}(S)\}$.
Now since, by Lemma \ref{lema2}, $\operatorname{F}(T) \geq d \operatorname{F}(S)$, we conclude both that $\operatorname{F}(T) = d \operatorname{F}(S)$ and that $\max \operatorname{M}_d(S) $ is finite.
\end{proof}

We note that the converse of the first statement in the above proposition is not true. 

\begin{example}\label{Ex9I}
Let $S = \langle 3,4,5 \rangle$ and $d = 3$. If $T=  \langle 4,7,9,10 \rangle$ and $T' = \langle 4,5,7 \rangle$, then one can see that $T \subsetneq T'$, that $\frac{T}d = \frac{T'}d = S$ and that $\operatorname{F}(T) = \operatorname{F}(T') = d \operatorname{F}(S) = 6.$ Therefore, $T \not\in \max \operatorname{M}_d(S) $ although has the minimum possible Frobenius number among the elements of $\operatorname{M}_d(S)$. 

Moreover, one can easily check that $\max \operatorname{M}_d(S)  = \{\langle 4,5,7 \rangle\}$ because there are only four numerical semigroups with Frobenius number equal to $6$: \[\langle 4,5,7 \rangle, \langle 4,7,9,10 \rangle, \langle 5,7,8,9,11 \rangle\ \text{and}\ \langle 7,8,9,10,11,12,13 \rangle,\] all of them contained in $\langle 4,5,7 \rangle$.
\end{example}

Once we know that $\max \operatorname{M}_d(S)$ is a finite set, we can use \cite[Algorithm 5.9]{aequationes} to compute it, simply observing, by Lemma \ref{lema2}, that \[\max(\operatorname {M} _d(S)) = \max \{T \in \mathscr{S} \mid d\, S \subseteq T\ \text{and}\ T \cap (d (\mathbb{N} \setminus S)) = \varnothing\}.\] 

\begin{remark}\label{Rem10I}
Since $\max \operatorname{M}_d(S)$ is contained in the set of numerical semigroups with Frobenius number $d \operatorname{F}(S)$, we can use the GAP\footnote{The GAP~Group. GAP -- Groups, Algorithms, and Programming, Version 4.12.2, 2022 [online]. Website {https://www.gap-system.org} [accessed 09 February 2024].} package \texttt{NumericalSgps}\footnote{Delgado M, {Garc\'{\i}a S\'anchez} PA, Morais H. NumericalSgps, a package for numerical semigroups, Version 1.3.1, 2022 [online]. Website {https://gap-packages.github.io/numericalsgps} [accessed 09 February 2024].} to compute $\max \operatorname{M}_d(S)$ for given $S$ and $d$. This method is not very efficient since it requires the computation of all numerical semigroups with Frobenius number $d \operatorname{F}(S)$.

For example, if we want to compute $\max \operatorname{M}_d(S)$ for $S = \langle 3,5,7 \rangle$ and $d=3$ we can use the following GAP code:
\begin{verbatim}
   LoadPackage("NumericalSgps");
   S:=NumericalSemigroup(3,5,7);
   d:=Int(3);
   FS:=FrobeniusNumber(S);
   L:=NumericalSemigroupsWithFrobeniusNumber(d*FS);;
   L:=Filtered(L,T->QuotientOfNumericalSemigroup(T,d)=S);;
   M:=[];
   for T in L do
      aux:=Difference(L,[T]);
      control:=Set(aux,R->IsSubsemigroupOfNumericalSemigroup(R,T));
      if control=[false] then Append(M,[T]); fi;
   od;
   M;
   List(M,i->MinimalGenerators(i));
\end{verbatim}
In this way, we obtain that $\max \operatorname{M}_d(S)$ is equal to
\[
\{ 
\langle 5, 8, 9, 11 \rangle, \langle 7, 8, 9, 10, 11, 13 \rangle
\}
\]
for $S = \langle 3,5,7 \rangle$ and $d=5$.
\end{remark}

Now, given $S \in \mathscr{S}, d \in \mathbb{N} \setminus \{0\}$ and $T \in \operatorname{M}_d(S)$, we write $\Theta_S^d(T)$ for the element of $\max \operatorname{M}_d(S) $ produced by the sequence \eqref{ecu1} for $T_0 = T$, and define the map \begin{equation}\label{Theta}\Theta_S^d : \operatorname{M}_d(S) \longrightarrow \max \operatorname{M}_d(S) ;\ T \mapsto \Theta_S^d (T) \end{equation} which is clearly surjective, giving rise the equivalence relation 
\[T \sim T' \Longleftrightarrow \Theta_S^d(T) = \Theta_S^d(T')\] on $\operatorname{M}_d(S)$, so that the natural map from $\max \operatorname{M}_d(S)$ to $\operatorname{M}_d(S)/\sim$ is a bijection. Therefore, to compute $\operatorname{M}_d(S)$, it is enough to know what the fibers of $\Theta_S^d$ are like.

Before dealing with the fibers of $\Theta_S^d$, let us look at an interesting result that relates the irreducibility of $S$ to that of the elements in $\max \operatorname{M}_d(S)$.

\begin{remark}\label{RemIrr}
Recall that a numerical semigroup is \emph{irreducible} if it cannot be written as the intersection of two numerical semigroups properly containing it. Recall also that a numerical semigroup is symmetric (pseudo-symmetric, resp.) if it is irreducible with odd (even, resp.) Frobenius number (see \cite[Chapter 4]{libro} for more details).
\end{remark}

\begin{proposition}
Let $S$ be a numerical semigroup and $d \in \mathbb{N} \setminus \{0\}$. Then the following are equivalent.
\begin{enumerate}[(1)]
\item $S$ is irreducible.
\item Every $T \in \max \operatorname{M}_d(S)$ is irreducible. 
\item Some $T \in \max \operatorname{M}_d(S)$ is irreducible. 
\end{enumerate}
\end{proposition}

\begin{proof}
First, let us prove (1) $\Rightarrow$ (2). So, suppose that $S$ is irreducible and let $T \in \max \operatorname{M}_d(S)$. By Lemma \ref{lema2}, $\operatorname{F}(T) = d \operatorname{F}(S)$. If $T$ is not irreducible, then, by \cite[Theorem 4.2]{libro}, there exists a numerical semigroup $T'$ with Frobenius number $d \operatorname{F}(S)$ strictly containing $T$. Now, since $\frac{T'}d$ is a numerical semigroup with Frobenius number $\operatorname{F}(S)$ containing $S$, by \cite[Theorem 4.2]{libro} again, $S = \frac{T'}d$. Therefore $T \subsetneq T' \in \operatorname{M}_d(S)$, a contradiction with the maximality of $T$.

The implication (2) $\Rightarrow$ (3) is clear.

Finally, if there exists $T \in \operatorname{M}_d(S)$ irreducible with Frobenius number $d \operatorname{F}(S)$, then $S$ is irreducible. Indeed, if $x \not\in S$ and $x \neq \operatorname{F}(S)/2$, then $d x \not\in T$ and $d x \neq d \operatorname{F}(S)/2 = \operatorname{F}(T)$; therefore, by \cite[Proposition 4.4]{libro}, $d \operatorname{F}(S) - d x \in T$. Thus, $\operatorname{F}(S) - x \in S$ and, by \cite[Proposition 4.4]{libro} again, we conclude that $S$ is irreducible which proves (3) $\Rightarrow$ (1).
\end{proof}

This result allows us to update the code given in Observation \ref{Rem10I} by replacing function
\begin{center}
\texttt{NumericalSemigroupsWithFrobeniusNumber}
\end{center}
by \texttt{IrreducibleNumericalSemigroupsWithFrobeniusNumber} 
when $S$ is irreducible (as it is actually the case in Remark \ref{Rem10I}), thus obtaining a faster computing time.

\section{The fibers of $\Theta_S^d$}\label{Sect3}

Throughout this section, $S \neq \mathbb{N}$ and $d$ are a numerical semigroup and a positive integer, respectively. 

Let $\theta_S^d(-)$ and $\Theta_S^d(-)$ be defined as in \eqref{theta} and \eqref{Theta}, respectively.
We begin by observing that $\Theta_S^d$ has a finite number of fibers of (possibly) infinite cardinality. In fact, since by \cite{swanson} $\operatorname{M}_d(S)$ has infinite cardinality, there exists at least one $T \in \max \operatorname{M}_d(S) $ such that $ (\Theta_S^d)^{-1}(T)$ has infinite cardinality. Regardless that there may exist fibers of finite cardinality as shown in Example \ref{Ex3.3}.

Given $R \in \max \operatorname{M}_d(S) $, we arrange the elements in the fiber of $R, (\Theta_S^d)^{-1}(R)$, in the form of a rooted tree  with root $R$ where the vertices are the elements of $(\Theta_S^d)^{-1}(R)$ and there is an edge joining the vertices $T$ and $T'$ if and only if $T' = T \cup \{\theta_S^d(T)\}$. We denote this tree by $\mathcal{G}^S_d(R)$, or simply by $\mathcal{G}(R)$ if there is no possibility of confusion.

The following proposition characterizes the set of children of a vertex in $\mathcal{G}(R)$ for a given $R \in \max \operatorname{M}_d(S) $.

\begin{proposition}\label{Prop10I}
Let $T \in \operatorname{M}_d(S)$. If $R = \Theta_S^d(T),$ then the set of children of $T$ in $\mathcal{G}(R)$ is equal to \begin{equation}\label{children}\{T \setminus \{x\} \subset \mathbb{N} \mid x \in \operatorname{msg}(T), x \not\in d\, S\ \text{and}\ \theta_S^d(T \setminus \{x\}) = x\}.\end{equation}
\end{proposition}

\begin{proof}
First of all, we recall that the necessary and sufficient condition for $T \setminus \{x\}$ to be a numerical semigroup is that $x \in \operatorname{msg}(T)$ (see, e.g. \cite[Exercise 2.1]{libro}). So, if $T'$ is a child of $T$ in $\mathcal{G}(R)$, that is, $T = T' \cup \{\theta_S^d(T')\}$ for some $T' \in (\Theta_S^d)^{-1}(R)$, then $T' = T \setminus \{\theta_S^d(T')\}$; consequently, $\theta_S^d(T') \in \operatorname{msg}(T)$ and $\theta_S^d(T \setminus \{\theta_S^d(T')\}) = \theta_S^d(T')$. Moreover, since $T \in \operatorname{M}_d(S)$, by Lemma \ref{lema2}, we have that $\theta_S^d(T') \not\in d\, S$, and we conclude that $T'$ belongs to \eqref{children}. Conversely, if $T'$ belongs to  \eqref{children}, by our initial observation, $T'$ is a numerical semigroup; moreover, since $T' \subsetneq T' \cup \{x\} = T$ and $x \not\in d\, S$, by Lemma \ref{lema2}, we have that $T' \in \operatorname{M}_d(S)$. Finally, since $T = T' \cup \{\theta_S^d(T')\},$ then $\Theta_S^d(T) = \Theta_S^d(T') = R$ by construction, and we are done.
\end{proof}

Note that the conditions $x \in \operatorname{msg}(T)$ and $x \not\in d\, S$ in \eqref{children} are equivalent to $T \setminus \{x\} \in \operatorname{M}_d(S)$, for any $T \in (\Theta_S^d)^{-1}(R)$ and $R \in \max\operatorname{M}_d(S)$. So to have an effective way to check whether $T \setminus \{x\}$ is a child of $T$ in $\mathcal{G}(\Theta_S^d(T))$, we need a criterion to decide whether $\theta_S^d(T \setminus \{x\}) = x$ or not. To do this, we need a few more concepts and results.

Following the terminology introduced in \cite{jpaa}, an integer number $z$ is a \emph{pseudo-Frobenius number of $T \in \mathscr{S}$} if $z \not\in T$ and $z + x \in T$ for every $x \in T \setminus \{0\}$. We write $\operatorname{PF}(T)$ for the set of pseudo-Frobenius numbers of $T \in \mathscr{S}$.

\begin{proposition}\label{Prop12}
If $T \in \operatorname{M}_d(S)$ and $T \not\in \max \operatorname{M}_d(S)$, then \[\theta_S^d(T) = \max\{z \in \operatorname{PF}(T) \mid 2z \in T\ \text{and}\ z \not\in d\, (\mathbb{N} \setminus S)\}.\]
\end{proposition}

\begin{proof}
Set $\theta := \max\{z \in \operatorname{PF}(T) \mid 2z \in T\ \text{and}\ z \not\in d\, (\mathbb{N} \setminus S)\}$. On the one hand, since $\theta_S^d(T) = \max\{x \in \mathbb{N} \setminus T \mid T \cup \{x\} \in \operatorname{M}_d(S)\}$, we have that $\theta_S^d(T) + x \in T,$ for every $x \in T \setminus \{0\}$ and that $2\, \theta_S^d(T) \in T$; moreover, by Lemma \ref{lema2}, $\theta_S^d(T) \not\in d\, (\mathbb{N} \setminus S)$. Therefore, $\max\{x \in \mathbb{N} \setminus T \mid T \cup \{x\} \in \operatorname{M}_d(S)\}$ is non-empty and  $\theta_S^d(T) \leq \theta.$ On the other hand, since $\theta \in \operatorname{PF}(T)$ and $2 \theta \in T$, we have that $T \cup \{\theta\}$ is a numerical semigroup; moreover, since $T \in \operatorname{M}_d(S)$ and $\theta \not\in d\, (\mathbb{N} \setminus S)$, we have that $d\, (\mathbb{N} \setminus S) \subseteq \mathbb{N} \setminus (T \cup \{\theta\}) \subseteq \mathbb{N} \setminus T \subseteq \mathbb{N} \setminus d\, S$. Therefore, by Lemma \ref{lema2}, $T \cup \{\theta\} \in M_d(S)$ and, consequently, $\theta \leq \theta_S^d(T)$.
\end{proof}

Using Proposition \ref{Prop12}, we can show that there are numerical semigroups $T$ which are leaves of $\mathcal{G}(\Theta_S^d(T))$ (that is, that have no children) for every $d \in \mathbb{N} \setminus \{0\}$ and $S = \frac{T}d$.

\begin{example}\label{Ex3.3}
Let $T = \langle 6,9,11 \rangle$. By direct computation, one can check that $\operatorname{PF}(T \setminus \{6\}) = \{6,9,25\}, \operatorname{PF}(T \setminus \{9\}) = \{9,16,25\}$ and $\operatorname{PF}(T \setminus \{9\}) = \{11,14,25\}$. So, for every $d \in \mathbb{N}\setminus \{0\}$ and $S = \frac{T}d$, we have that $T$ is a leaf of $\mathcal{G}(\Theta_{S}^d(T))$ by Proposition \ref{Prop12}. Moreover, if $d = 5$, then $S = \langle 3,4 \rangle$ and $T \in \max \operatorname{M}_d(S)$; therefore, in this case, $(\Theta_{S}^d)^{-1}(T)=\{T\}$.
\end{example}

Let us see that in most cases the maximum of the proposition $\theta_S^d(T)$ can be determined immediately.

\begin{corollary}\label{Cor13I}
If $T \in \operatorname{M}_d(S)$ and $\operatorname{F}(T)$ is not divisible by $d$, then $\theta_S^d(T) = \operatorname{F}(T)$.
\end{corollary}

\begin{proof}
By Proposition \ref{Prop12}, it suffices to observe that $\operatorname{F}(T) = \max \operatorname{PF}(T)$ and that $2 \operatorname{F}(T) \in T$.
\end{proof}

\begin{corollary}\label{Cor14I}
Let $T \in \operatorname{M}_d(S)$ and $x \in \operatorname{msg}(T)$ be such that $x \not\in d\, S$. If $x > \operatorname{F}(T)$, then $\theta_S^d(T \setminus \{x\}) = x$.
\end{corollary}

\begin{proof}
Since $x > \operatorname{F}(T)$, we have that $\operatorname{F}(T \setminus \{x\}) = x$. Moreover, by Lemma \ref{lema2}, $x > \operatorname{F}(T) \geq d \operatorname{F}(S)$; in particular, $x \not\in d(\mathbb{N} \setminus S)$.  So, since $x = \operatorname{F}(T \setminus \{x\}) = \max \operatorname{PF}(T \setminus \{x\})$, we have that $x = \max\{z \in \operatorname{PF}(T \setminus \{x\}) \mid 2z \in T \setminus \{x\}\ \text{and}\ z \not\in d (\mathbb{N} \setminus S)\}$ and, by Proposition \ref{Prop12}, we conclude that $\theta_S^d(T \setminus \{x\}) = x$.
\end{proof}

The following result characterizes the elements of $\operatorname{M}_d(S)$ with minimal Frobenius number (see Lemma \ref{lema2}).

\begin{lemma}\label{Lema11I}
If $T \in \operatorname{M}_d(S)$, then $\operatorname{F}(T)$ is divisible by $d$ if and only if $\operatorname{F}(T) = d \operatorname{F}(S)$.
\end{lemma}

\begin{proof}
If $\operatorname{F}(T) = d \operatorname{F}(S)$, then it is obvious that $\operatorname{F}(T)$ is divisible by $d$. Conversely, assume the opposite: $\operatorname{F}(T) \neq d \operatorname{F}(S)$; in particular, by Proposition \ref{prop8I}, $T \not\in \max \operatorname{M}_d(S) $. So, by Lemma \ref{lema7}, $T' = T \cup \{\theta_S^d(T)\}$ is a well-defined element of $\operatorname{M}_d(S)$. Therefore, by Proposition \ref{Prop10I}, there exists $x \in \operatorname{msg}(T')$ with $x \not\in d\, S, \theta_S^d(T' \setminus \{x\}) = x$ and $T' \setminus \{x\} = T$. If $x < \operatorname{F}(T')$, then $\operatorname{F}(T') = \operatorname{F}(T)  \neq d \operatorname{F}(S)$. So, we may replace $T$ by $T'$ and repeat the above arguments. Thus, without loss of generality, we may suppose that $x > \operatorname{F}(T')$. In this case, $x = F(T' \setminus \{x\}) = F(T)$ is divisible by $d$ by hypothesis, that is, $x = d y$ for some $y \in \mathbb{N} \setminus \{0\}$. Now, since $x \not\in d\, S$, we have that $y \not\in S$. Therefore, $y \in \frac{T'}d \neq S$, a contradiction. 
\end{proof}

Recall that, even though  $\operatorname{F}(T) = d \operatorname{F}(S)$ for every $T \in \max \operatorname{M}_d(S) $, we already know that the condition $\operatorname {F}(T) = d \operatorname{F}(S),\ T \in \operatorname{M}_d(S)$, it does not imply $T \in \max(\operatorname{M }_d (S))$ (see Example \ref{Ex9I}).

The following result give an easier characterization of the set of children of $T \in \operatorname{M}_d(S)$ in $\mathcal{G}(\Theta_S^d(T))$, when $\operatorname{F}(T)$ is not minimal.

\begin{proposition}
Let $T \in \operatorname{M}_d(S)$ and $R = \Theta_S^d(T)$. If $\operatorname{F}(T) \neq d \operatorname{F}(S)$, then the set of children of $T$ in $\mathcal{G}(R)$ is equal to \[\{T \setminus \{x\} \subset \mathbb{N} \mid x \in \operatorname{msg}(T), x \not\in d\, S\ \text{and}\ x > \operatorname{F}(T)\}.\] In particular, if $\operatorname{F}(T) \neq d \operatorname{F}(S)$ and $x < \operatorname{F}(T)$ for every $x \in \operatorname{msg}(T)$, then $T$ is a leaf of $\mathcal{G}(R)$.
\end{proposition}

\begin{proof}
Let $x \in \operatorname{msg}(T) \setminus d\, S$. On the one hand, if $x < \operatorname{F}(T),$ then $\operatorname{F}(T \setminus \{x\}) = \operatorname{F}(T) \neq d \operatorname{F}(S)$. Thus, by Lemma \ref{Lema11I}, $\operatorname{F}(T \setminus \{x\})$ is not divisible by $d$ and, by Corollary \ref{Cor13I}, $\theta_S^d(T \setminus \{x\}) = \operatorname{F}(T \setminus \{x\}) = \operatorname{F}(T) > x$; in particular, $\theta_S^d(T \setminus \{x\}) \neq x$.  On the other hand, if $x > \operatorname{F}(T)$, then, by Corollary \ref{Cor14I}, $\theta_S^d(T \setminus \{x\}) = x$. Therefore,  $\theta_S^d(T \setminus \{x\}) = x$ if and only if $x > \operatorname{F}(T)$. Now, by Proposition \ref{Prop10I}, we are done.
\end{proof}

The condition $\operatorname{F}(T) \neq d \operatorname{F}(S)$ is necessary in the proposition above.

\begin{example}
Let $S = \langle 2,3 \rangle$ and $d = 11$. If $T = \langle 5,7,8,9 \rangle$, then $\frac{T}{11} = S$. In this case, $\operatorname{F}(T) = 11 = 11 \cdot \operatorname{F}(S)$ and both $T \setminus \{8\} = \langle 5,7,9,13 \rangle$ and $T \setminus \{9\} = \langle 5,7,8 \rangle$ are children of $T$ in $\mathcal{G}(\Theta_S^d(T))$. So, $T$ is not a leaf of $\mathcal{G}(\Theta_S^d(T))$.
\end{example}

The following example shows that leaves can actually appear in the case $\operatorname{F}(T) = d \operatorname{F}(S)$.

\begin{example}
Let $S = \langle 3,5,7 \rangle$ and $d = 3$. The numerical semigroup $T = \langle 5,8,9 \rangle$ belongs to $\operatorname{M}_d(S)$ and has Frobenius number $12 = 3 \operatorname{F}(S)$. By Proposition \ref{Prop12}, we have that $T$ is a leaf of $\mathcal{G}(\Theta_S^d(T))$, because $\operatorname{PF}(T \setminus \{5\}) = \{5,6,7,11,12\}, \operatorname{PF}(T \setminus \{8\}) = \{4,8,11,12\}$ and $\operatorname{PF}(T \setminus \{9\}) = \{9,11,12\}$.
\end{example}

\section{On numerical semigroups with full quotient rank}\label{Sect6}

Let $S$ be a numerical semigroup. In this section we give a sufficient condition for $S$ to have full quotient rank, that is, for $S$ to verify the following equality: \[\min\{\operatorname{e}(T) \mid T\ \text{is a multiple of}\ S\} = \operatorname{e}(S).\] To do this, let us first recall the notion of Ap\'ery set of $S$ with respect to $x \in S \setminus \{0\}$.

\begin{definition}
Let $S$ be a numerical semigroup and let $x$ be one of its nonzero elements. The \emph{Ap\'ery set of $S$ with respect to $x$} is 
\[\operatorname{Ap}(S,x) := \{y \in S \mid y - x \not\in S\}.\]
\end{definition}

\begin{proposition}\label{Prop_full_rank}
Let $S$ be a numerical semigroup such that $\operatorname{msg}(S) = \{a_1, \ldots, a_e\}$. If 
\begin{equation}\label{ecu_fullrank}
\sum_{\substack{j=1\\ j \neq i}}^e a_j \in \operatorname{Ap}(S,a_i)\ \text{for every}\ i \in \{1, \ldots, e\},
\end{equation} then $S$ has full quotient rank.
\end{proposition}

\begin{proof}
Suppose instead that $S$ does not have full quotient rank. Then, by \cite[Corollary 2.2]{australia}, there exists $J \subseteq \{1, \ldots, e\}$ such that $\sum_{j \in J} a_j = \sum_{j \in \{1, \ldots, e\} \setminus J} u_j a_j$ for some $u_j \in \mathbb{N}, \ j \in \{1, \ldots, e\} \setminus J,$ not all zero. Let $i \in \{1, \ldots, e\} \setminus J$ be such that $u_i \neq 0$. Since \begin{align*}\sum_{\substack{j=1\\ j \neq i}}^e a_j & = \sum_{j \in J} a_j + \sum_{j \in \{1, \ldots, e\} \setminus (J \cup \{i\})} a_j = \sum_{j \in \{1, \ldots, e\} \setminus J} u_j a_j + \sum_{j \in \{1, \ldots, e\} \setminus (J \cup \{i\})} a_j\\ & = u_i a_i + \sum_{j \in \{1, \ldots, e\} \setminus (J \cup \{i\})} (u_j+1) a_j,\end{align*} we have that \[\sum_{\substack{j=1\\ j \neq i}}^e a_j - a_i = (u_i-1) a_i + \sum_{j \in \{1, \ldots, e\} \setminus (J \cup \{i\})} (u_j+1) a_j \in S,\] contradicting the hypothesis.
\end{proof}

It is worth asking whether there are numerical semigroups that satisfy the condition \eqref{ecu_fullrank}. The answer is yes, as the following example shows.

\begin{example}
Let $S = \langle 21,24,25,31 \rangle$. Using the function 
\begin{center}
\texttt{AperyListOfNumericalSemigroupWRTElement}    
\end{center}
included in the GAP\footnote{The GAP~Group. GAP -- Groups, Algorithms, and Programming, Version 4.12.2, 2022 [online]. Website {https://www.gap-system.org} [accessed 09 February 2024].} package \texttt{NumericalSgps}\footnote{Delgado M, {Garc\'{\i}a S\'anchez} PA, Morais H. NumericalSgps, a package for numerical semigroups, Version 1.3.1, 2022 [online]. Website {https://gap-packages.github.io/numericalsgps} [accessed 09 February 2024].} , we can easily check that 
\[
24+25+31 = 80 \in \operatorname{Ap}(S,21),\quad 21+25+31 = 77 \in \operatorname{Ap}(S,24),
\]
\[
21+24+31 = 76 \in \operatorname{Ap}(S,25) \quad \text{and}\quad 21+24+25 = 70 \in \operatorname{Ap}(S,31).
\]
Therefore, by Proposition \ref{Prop_full_rank}, we conclude that $S$ has full quotient rank.
\end{example}

Note that the semigroup in the example above is not one of those shown in \cite[Theorem 3.1]{australia}. Let us introduce a new family of numerical semigroups with full quotient rank. 

In \cite{unique}, \emph{numerical semigroups having unique Betti element} are introduced. By \cite[Theorem 12]{unique}, we have that a numerical semigroup of embedding dimension $e \geq 2$ has unique Betti number if and only if there exist relatively prime integers $c_1, \ldots, c_e$ greater than one such that \[\operatorname{msg}(S) = \left\{\prod_{\substack{j=1\\ j \neq i}}^e c_j \mid i \in \{1, \ldots, e\}\right\}.\]

\begin{corollary}\label{Cor41I}
If $S$ is a numerical semigroup with a unique Betti element, then it has full quotient rank.
\end{corollary}

\begin{proof}
Let $\operatorname{msg}(S) = \{a_1, \ldots, a_e\}$. Since $S$ is a numerical semigroup having a unique Betti element, by \cite[Theorem 12]{unique}, there exist relatively prime integers $c_1, \ldots, c_e$ greater than one such that $a_i = \prod_{\substack{j=1\\ j \neq i}}^e c_j,\ i = 1, \ldots, e$.

Let $\Bbbk$ be a field and consider the $\Bbbk-$algebra homomorphism \[\varphi : \Bbbk[x_1, \ldots, x_e] \longrightarrow \Bbbk[S] := \bigoplus_{s \in S} \Bbbk \{\chi^s\} \] such that $\varphi(x_i) = \chi^{a_i},\ i = 1, \ldots, e$. By \cite[Corollary 10]{unique}, a system of generators of $\ker \varphi$ is $\{x_i^{c_i} - x_j^{c_j} \mid 1 \leq i < j \leq e\}$. Now, by \cite[Theorem 6]{MOT}, we have that \[\operatorname{Ap}(S,a_i) = \left\{\sum_{\substack{j=1\\ j \neq i}}^e u_j a_j \mid 0 \leq u_j \leq c_j,\ j \in \{1, \ldots, e\} \setminus \{i\} \right\},\] for every $i \in \{1, \ldots, e\}$. In particular, $\sum_{\substack{j=1\\ j \neq i}}^e a_j \in \operatorname{Ap}(S,a_i),$ for every $i \in \{1, \ldots, e\}$ and, by Proposition \ref{Prop_full_rank}, we are done.
\end{proof}

\begin{remark}
Using \cite[Proposition 3.10]{uf} and the same argument as that applied in the proof of Corollary \ref{Cor41I}, one can prove that universally free numerical semigroups (introduced in \cite{uf}) also have full quotient rank.
\end{remark}

Finally, we delve into some questions related to Proposition \ref{Prop_full_rank}. Of course, the most immediate question is whether the condition \eqref{ecu_fullrank} is also necessary. So far we do not have sufficient evidence to formulate a conjecture, since all examples of numerical semigroups having full rank occur as a by-product of the results in \cite{australia}. 

Apart from this, an interesting consequence of \eqref{ecu_fullrank} is the following.

\begin{proposition}\label{lema_multi}
Let $S$ be a numerical semigroup such that $\operatorname{msg}(S) = \{a_1 < \cdots < a_e\}$. If 
$\sum_{\substack{j=1\\ j \neq i}}^e a_j \in \operatorname{Ap}(S,a_i)$ for every $i \in \{1, \ldots, e\}$, then $a_1 \geq 2^{e-1}$.
\end{proposition}

\begin{proof}
Clearly, $\sum_{j=2}^e a_j \in \operatorname{Ap}(S,a_1)$ implies $a_J := \sum_{j \in J} a_j \in \operatorname{Ap}(S,a_1)$ for every non-empty $J \subseteq \{2, \ldots, e\}$. If there exist two different non-empty subsets $J$ and $J'$ of $\{2, \ldots, e\}$ with $a_J = a_{J'}$, then there exists $i \in J' \setminus J$ such that \[\sum_{\substack{j=1\\ j \neq i}}^e a_j = a_J + \sum_{j \in \{1, \ldots, e\} \setminus (J \cup \{i\})} a_j = a_{J'} + \sum_{j \in \{1, \ldots, e\} \setminus (J \cup \{i\})} a_j \not\in \operatorname{Ap}(S,a_i)\] which contradicts the hypothesis. Therefore, since $0 \in \operatorname{Ap}(S,a_1)$,  we conclude that there are at least $2^{e-1}$ different elements in $\operatorname{Ap}(S,a_1)$. Now, since the cardinality of $\operatorname{Ap}(S,a_1)$ is equal to $a_1$ (see, e.g. \cite[Lemma 2.4]{libro}), we are done.
\end{proof}

Recall that $\min(\operatorname{msg}(S))$ is called \emph{multiplicity of $S$}, denoted $\operatorname{m}(S)$. By \cite[Proposition 2.10]{libro}, $\operatorname{e}(S) \leq \operatorname{m}(S)$. Therefore, if $\operatorname{e}(S) > 2$ and $S$ verifies the condition \eqref{ecu_fullrank}, then $S$ does not have a maximum embedding dimension according to Lemma \ref{lema_multi}. In \cite[Theorem 4.5]{australia} it is shown that numerical semigroups with maximum embedding dimension do not have full quotient rank. Therefore, if the converse of Proposition \ref{Prop_full_rank} is true, Lemma \ref{lema_multi} would provide an alternative proof of this fact.

\section*{Data availability}\mbox{}\par
\noindent Data sharing is not applicable to this article, as no datasets were generated or analyzed during the present work.

\section*{Declarations}\mbox{}\par 
\noindent The authors are partially supported by Proyecto de Excelencia de la Junta de Andaluc\'{\i}a (ProyExcel\_00868). The second author is partially supported by grant PID2022-138906NB-C21 funded by {MCIN/AEI/10.13039/5011000 11033} and by ERDF ``A way of making Europe''.

\medskip
\noindent The authors have no relevant financial or non-financial interests to disclose.

\medskip
\noindent
All authors have contributed equally in the development of this work.

\end{document}